\documentclass{article}
\usepackage{amsmath}
\usepackage{amsmath,amssymb, amsthm, latexsym, mathrsfs,amsfonts, epsfig, color,graphicx,}

\renewcommand{\a}{\alpha}

\newcommand{\e}{\epsilon}

\renewcommand{\l}{\left}
\renewcommand{\r}{\right}

\newtheorem{theorem}{Theorem}
\newtheorem{lemma}{Lemma}
\newtheorem{remark}{Remark}

\newtheorem{definition}{Definition}

\begin{document}

\title{Effective Approximation for a Nonlocal Stochastic Schr\"{o}dinger Equation with Oscillating Potential \thanks{This work was partly supported by the National Natural Science Foundation of China (NSFC)$11771161$, $11971184$ and $11771449.$}}
\author{Li Lin \footnote{School of Mathematics and Statistics and Center for Mathematical Sciences, Huazhong University of Science and Technology,Wuhan, 430074, China. E-mail address: linli@hust.edu.cn}
\and  Meihua Yang \footnote{School of Mathematics and Statistics and Center for Mathematical Sciences, Huazhong University of Science and Technology,Wuhan, 430074, China. E-mail address: yangmeih@hust.edu.cn}
\and Jinqiao Duan \footnote{Corresponding author. Department of Applied Mathematics, Illinois Institute of Technology, Chicago, Illnois 60616, USA. E-mail address: duan@iit.edu.}}

\maketitle
\begin{abstract}
We study the effective approximation for a nonlocal stochastic Schr\"{o}dinger equation with a rapidly oscillating, periodically time-dependent potential. We use the natural diffusive scaling of heterogeneous system and study the limit behaviour as the scaling parameter tends to 0.  This is motivated by data assimilations with non-Gaussian uncertainties. The nonlocal operator in this stochastic partial differential
equation is the generator of a non-Gaussian L\'evy-type process (i.e., a class of anomalous diffusion processes), with non-integrable jump kernel.
With help of  a  two-scale convergence technique, we establish effective approximation for this nonlocal stochastic partial differential equation. More precisely,  we show that a nonlocal stochastic Schr\"{o}dinger equation has a nonlocal effective equation. We show that it approximates the orginal stochastic Schr\"{o}dinger equation weakly in a Sobolev-type space and strongly in $L^2$ space. In  particular,  this effective approximattion  holds when the nonlocal operator is the fractional Laplacian.
\end{abstract}

\textbf{Keywords:}  Effective approximation, Nonlocal Laplace operator, Schr\"{o}dinger equation, Effective dynamics.

\textbf{Mathematics Subject Classification:} 60H15, 35B27, 80M40, 26A33
\section{\bf Introduction}

The effective approximation of stochastic partial differential equations has attracted a lot of attention recently \cite{2,Duan,Otto}, due to its importance in valid mathematical modeling and efficient simulation.
The Schr\"{o}dinger equation is the fundamental equation in quantum physics for describing quantum mechanical behaviors. It quantifies  the wave function of a quantum system evolving over time. For the homogenization of deterministic Schr\"{o}dinger equations, there are two different scalings. One is the semi-classical scaling  \cite{Buslaev,Dimassi,Gerard}, and the other one is the typical scaling of homogenization \cite{Allaire,Allairetwo}.


 In the path integral approach \cite{Feynman} to quantum physics,   the integral over the Brownian trajectories leads to the   usual   (local) Schr\"{o}dinger equation \cite{Albeverio}.  Recent works on the path integrals over
  the  L\'{e}vy paths  (e.g., \cite{Piotr}) lead to    nonlocal Schr\"{o}dinger equations.  More physical investigations on fractional or nonlocal generalization of the Schr\"{o}dinger equations may be
  found in, for example, \cite{Dinh, Jiarui, Obrezkov, Yang, Zaba}.

As   random disturbances may affect the qualitative behaviors drastically and result in new properties for this quantum model, stochastic Schr\"{o}dinger equations have attracted attentions recently   (e.g.,  \cite{De,Gautier,Keller,Pellegrini, Zastawniak}).

In this paper, we will establish an effective approximation for a nonlocal stochastic Schr\"{o}dinger equation with a typical scaling and an oscillating potential. For stochastic homogenization problems,  a   two-scale convergence technique  \cite{Bourgeat,Casado,Heida,Zhikov}
is available.

Specifically, we consider the effective approximation for the following  \emph{nonlocal} stochastic Schr\"{o}dinger equation (\textbf{heterogeneous system}) with a small  positive scale parameter $\epsilon$:
\begin{small}
\begin{equation}\label{1}
\begin{cases}
$$idu_\e=\l(\frac{1}{2}\mathcal{A}^\e u_\e+\epsilon^{(1-\a)/2}\mathcal{V}^\epsilon u_\e\r)dt+g(u_\e)dW_t+fdt, \; 0<t<T, \; x \in D=(-1,1),  \\
u_\epsilon(0,x)=h(x),\quad\quad x \in D,\\
u_\e(t,x)=0,  \quad\quad x \in D^c=\mathbb{R}\backslash D,$$
\end{cases}
\end{equation}
\end{small}
where $\a\in(1,2)$, and where $f\in L^2\l(0,T;L^2(D)\r)$.
The function $\mathcal{V}^\e(t,x)=\mathcal{V}\l(\frac{t}{\e},\frac{x}{\e}\r)$ is a real potential and
 $W(t)$ is a one dimensional Brownian motion defined on a complete probability space $(\Omega, \mathcal{F},\mathbb{P})$. The function $g$ is noise intensity and satisfies the Lipschitz condition.
We consider the nonlocal operator
$$\mathcal{A}^\e u=\mathcal{D}\l(\Theta^\e\mathcal{D}^*u\r),$$
where $\Theta^\e(x,z)=\Theta\l(\frac{x}{\e},\frac{z}{\e}\r)$ is positive, bounded function, and the linear operator $\mathcal{D}$ and its adjoint operator $\mathcal{D}^*$ are defined as follows.

Given functions $\beta(x,z)$ and $\gamma(x,z)$, 
the nonlocal divergence $\mathcal{D}$ on $\beta$ is defined as
$$\mathcal{D}(\beta)(x):=\int_{\mathbb{R}}(\beta(x,z)+\beta(z,x))\cdot\gamma(x,z)dz  \qquad \text{for}\; x\in \mathbb{R}.$$
For a function  $\phi(x)$, the adjoint operator $\mathcal{D^{*}}$ corresponding to $\mathcal{D}$ is the operator whose action on $\phi$ is given by
$$\mathcal{D^{*}}(\phi)(x,z)=-(\phi(z)-\phi(x))\gamma(x,z) \qquad \text{for} \; x,z\in \mathbb{R}.$$
The aim here is to investigate the limiting behaviour of $u_{\e}$ when $\varepsilon$ goes to zero, under the periodicity hypotheses on the coefficients
$\Theta$ and the potential $\mathcal{V}$, and the assumption that the mean value of $\mathcal{V}$ for the spatial variable is null.

Through this paper, we always take $\gamma(x,z)=(z-x)\frac{1}{|z-x|^{\frac{3+\a}{2}}}.$
As a special case, we set $\Theta$ to be $1$, we have
$$\frac{1}{2}\mathcal{D}\mathcal{D^{*}}=-(-\Delta)^{\a/2}.$$  The nonlocal Laplace operator $(-\Delta)^{\a/2}$ is defined as
$$(-\Delta)^{\a/2}u(x)=\int_{\mathbb{R}\backslash \{x\}}\frac{u(z)-u(x)}{|z-x|^{1+\a}}dz,$$
where the integral is in the sense of Cauchy principal value.

\begin{remark}
  Without loss of generality, we set $\Theta^{\e}(x,z)$ to be a symmetric function. In fact, we can define the symmetric and anti-symmetric parts of $\Theta$:
  $$\Theta^{\e}_s(x,z)=\frac{1}{2}(\Theta^{\e}(x,z)+\Theta^{\e}(z,x)) \quad \text{and} \quad \Theta^{\e}_a(x,z)=\frac{1}{2}(\Theta^{\e}(x,z)-\Theta^{\e}(z,x)).$$
From the fact that $$\mathcal{D}\l(\Theta^{\e}_a(x,z)\mathcal{D}^*u\r)=0,$$ we have $$\mathcal{D}\l(\Theta^{\e}(x,z)\mathcal{D}^*u\r)=\mathcal{D}\l(\Theta^{\e}_s(x,z)\mathcal{D}^*u\r).$$
\end{remark}

\begin{remark}
  For a function $\upsilon(x,y)$, we define
$$(\mathcal{D}^*_x\upsilon)(x,z,y)=-(\upsilon(z,y)-\upsilon(x,y))\gamma(x,z)$$
and
\[
\begin{split}
(\mathcal{D}_x\mathcal{D}^*_x\upsilon)(x,y)&=2\int_{\mathbb{R}}
-(\upsilon(z,y)-\upsilon(x,y))\gamma^2(x,z)dz\\
&=-2(-\Delta)^{\a/2}_x\upsilon(x,y).
\end{split}
\]
\end{remark}

Our purpose is to examine the convergence of  the solution  $u_\epsilon$  of (\ref{1}) in some probabilistic sense,  as $\epsilon\rightarrow0$,
and to specify the limit $\tilde{u}$. We will see that the limit process $\tilde{u}$ satisfies the following nonlocal stochastic partial differential equation (\textbf{effective system}):
\begin{align}\label{n7}
  \begin{cases}
    id\tilde{u}=\l(-\Xi_1(-\Delta)^{\a/2}\tilde{u}-\frac{\Xi_2}{2}\mathcal{D}|_{D}(\zeta)(x)-\Xi_3\zeta(x)\r)dt+g(\tilde{u})dW_t+fdt,\\
    \tilde{u}(x,t)=0,\quad\quad (x,t) \in D^c\times(0,T),\\
\tilde{u}(0)=h(x),  \quad\quad x \in D,$$
  \end{cases}
\end{align}
where
\begin{align*}
  &\Xi_1=\int_{Y\times N}\Theta(y,\eta)dydn,\\
  &\Xi_2=\int_{Y\times N \times Z}\Theta(y,\eta)\mathcal{D}_y^*\chi dydnd\tau,\\
  &\Xi_3=2\int_{Y\times Z}\mathcal{V}(y,\tau)\chi(y,\tau)dyd\tau,\\
  &\zeta(x)=\frac{1}{|D|}\int_{D}(\mathcal{D}^*\tilde{u})(x,z)dz,\\
  &\mathcal{D}|_{D}(\zeta)(x)=\int_{D}(\zeta(x)+\zeta(z))\gamma(x,z)dz,\\
\end{align*}
where the function $\chi$ is given by (\ref{654}). We set $Y=N=Z=(0,1)$ and consider $Y, N, Z$ as subsets of $\mathbb{R}_y, \mathbb{R}_\eta, \mathbb{R}_\tau $ respectively (the spaces of variables $y$, $\eta$ and $\tau$ respectively).
Through this paper, we always identify functions on $\mathbb{T}=(0,1)$ with their periodic extension to $\mathbb{R}$.

\textbf{Structure of this paper.} 
In Section $2$, we recall some function spaces and deal with the existence and uniqueness of the Schr\"{o}dinger equation. Then, in Section $3$, we prove the main theorem and derive the effective system.
\section{Preliminaries}
 We now briefly discuss the well-posedness for the heterogeneous equation (\ref{1}), and derive a few uniform estimates concerning the solution $u^\e$.
\subsection{\bf Function spaces}
We set $Y=N=Z=(0,1)$ and consider $Y, N, Z$ as subsets of $\mathbb{R}_y, \mathbb{R}_\eta, \mathbb{R}_\tau $ respectively (the spaces of variables $y$, $\eta$ and $\tau$ respectively).

Let us first recall that a function $u$ is said to be $Y\times N\times Z$-periodic if for each $k,l,m\in \mathbb{Z}$, we have $u(y+k,\eta+l,\tau+m)=u(y,\eta,\tau)$ almost everywhere with $y,\eta,\tau\in \mathbb{R}.$ The space of all $Y\times N\times Z$-periodic continuous complex functions on $\mathbb{R}_y\times\mathbb{R}_\eta\times\mathbb{R}_\tau$
is denoted by $\mathcal{C}_{per}(Y\times N\times Z)$, that of all $Y\times N\times Z$-periodic functions in $L^p\l(\mathbb{R}_y\times\mathbb{R}_\eta\times\mathbb{R}_\tau\r)(1\leq p\leq\infty)$ is
denoted by $L^p_{per}(Y\times N\times Z).$ $\mathcal{C}_{per}(Y\times N\times Z)$ is a Banach space under the supremum norm on $\mathbb{R}\times\mathbb{R}\times\mathbb{R},$ whereas $L^p_{per}(Y\times N\times Z)$ is a Banach space under the norm
$$\|u\|_{L^p_{per}(Y\times N\times Z)}=\l(\int_{Y\times N\times Z}\l|u(y,\eta,\tau)\r|^pdyd\eta d\tau\r)^{\frac{1}{p}}.$$

Let $\a\in(1,2)$, the classical fractional Sobolev space is
$$H^{\a/2}(D)=\l\{u\in L^2(D): \int_{D}\int_{D}\frac{|u(x)-u(z)|^2}{|x-z|^{1+\a}}dxdz<\infty\r\},$$
with the norm
$$ \|u\|_{H^{\a/2}(D)}^2=\|u\|_{L^2(D)}^2+\int_{D}\int_{D}\frac{|u(x)-u(z)|^2}{|x-z|^{1+\a}}dxdz.$$
We set $u|_{\mathbb{R}\backslash D}\equiv 0$. For the nonlocal operator,  we have
\[
\begin{split}
(\frac{1}{2}\mathcal{A}^\epsilon&u,u)_{L^2(D)}=\frac{1}{2}\l(\Theta^\epsilon(x,z)\mathcal{D}^*u(x,z),\mathcal{D}^*u(x,z)\r)_{L^2(\mathbb{R}\times\mathbb{R})}\\
&=\int_D\int_{D^c}\Theta^\epsilon(x,z)\frac{|u(x)|^2}{|z-x|^{1+\a}}dzdx+\frac{1}{2}\int_{D}\int_{D}\Theta^\epsilon(x,z)\frac{|u(x)-u(z)|^2}{|x-z|^{1+\a}}dxdz.
\end{split}
\]
Pose $\rho(x):=\int_{D^c}\frac{1}{|z-x|^{1+\a}}dz.$ Since the fact that $\Theta^\epsilon(x,z)$ is a positive, bounded function, we then can define a weighted fractional Sobolev space without considering the function $\Theta^\epsilon(x,y):$
$$H_\rho^{\a/2}(D):=\l\{u\in L^2(\mathbb{R}): u|_{\mathbb{R}\backslash D}\equiv 0, \|u\|_{H_\rho^{\a/2}(D)}< \infty \r\},$$
equipped with the norm
$$\|u\|_{H_\rho^{\a/2}(D)}:= \l(\int_D\rho(x)|u(x)|^2dx+\frac{1}{2}\int_{D}\int_{D}\frac{|u(x)-u(z)|^2}{|x-z|^{1+\a}}dxdz\r)^{\frac{1}{2}},$$
which immediately implies that $\l(-(-\Delta)^{\a/2} u,u\r)_{L^2(D)}=\|u\|^2_{H_\rho^{\a/2}(D)}.$


The space $H^{\a/2}_\#(Y)$ of $Y$-periodic functions $u\in H^{\a/2}$ such that $\int_{Y}u(y)dy=0$ will be interest in this study. Provided with the norm,
$$\|u\|_{H^{\a/2}_\#(Y)}=\l(\int_{Y}\int_{N}\frac{|u(y)-u(\eta)|^2}{|y-\eta|^{1+\a}}dyd\eta\r)^{\frac{1}{2}}$$

For $s<0$, we define $H^s(D)$ as the dual space of $H^{-s}(D)$.

\subsection{\bf Well-posedness}
Let $B^\e$ be the linear operator in $L^2(D)$
defined by\\
\begin{center}
$B^\e u=-i\mathcal{A}^\e u $ for all $u\in D(B^\e),$\\
\end{center}
with domain
$$D(B^\e)=\{\upsilon\in H^{\a/2}_\rho(D): \mathcal{A}^\e\upsilon\in L^2(D)\}.$$
Then, $B^\e$ is of dense domain, and skew-adjoint since $\mathcal{A^\e}$ is self-adjoint\cite{Caze}. Consequently, $B^\e$ is a $m$-dissipative operator in $L^2(D)$(Corollary 2.4.11 of \cite{ Caze}). It follows by the Hille-Yosida-Philips theorem that $B^\e$ is the generator of a contraction semigroup $(G_t^\e)_{t>0}$.

Now, let us check the existence and uniqueness for equation (\ref{1}). The abstract problem for equation (\ref{1}) is given by
\begin{equation}\label{2}
\begin{cases}
$$du_\e=\l(B^\e u_\e+F_\e(u_\e)\r)dt+g(u_\e)dW_t, \\
u_\epsilon(0,x)=h(x),$$
\end{cases}
\end{equation}
where $F_\e$ is defined in $L^2\l(0,T;L^2(D)\r)$ by
$$F_\e(\upsilon)(t)=-i\epsilon^{1-\a}\mathcal{V}^\epsilon \upsilon-if(t).$$
We can obtain the following lemma (Theorem $3.3$ of \cite{Gaw}).
\begin{lemma}\label{l1}
  Suppose $h\in L^2(D),$ $f\in C\l([0,T];L^2(D)\r)$ and for all $\e>0$,
  \begin{equation}\label{321}
    \l\|\e^{(1-\a)/2}\mathcal{V}^\e\r\|_{\mathcal{L}\l(H_\rho^{\a/2}(D), H_\rho^{-\a/2}(D)\r)}\leq \beta,
  \end{equation}
  where $\beta$ is a positive constant independent of $\epsilon$ and $\mathcal{L}\l(H_\rho^{\a/2}(D), H_\rho^{-\a/2}(D)\r)$ is the space of linear continuous mapping of
  $H_\rho^{\a/2}(D)$ into $ H_\rho^{-\a/2}(D).$
  We obtain the existence and uniqueness of mild solution $u_\e(t)\in C
  \l([0,T];L^2(\Omega\times D)\r).$
\end{lemma}
\begin{remark}
  For an illustrative example, if the potential $\mathcal{V}$ belongs to $C_{per}(Y)$ and verifies
\begin{equation}\label{421}
\int_{Y}\mathcal{V}(y)dy=0.
\end{equation}
  Then, the linear operator $\e^{(1-\a)/2}\mathcal{V}^\e$ vertifies (\ref{321}). Indeed, since $\mathcal{V}\in C_{per}(Y)$ and verifies (\ref{421}), the equation
\begin{equation*}
 \mathcal{D}_y(\Theta \mathcal{D}_y^*\varsigma)=\mathcal{V}
\end{equation*}
admits a unique solution $\varsigma$ in $H^{\a/2}_\#(Y).$ Let $\varsigma^\e(x)=\varsigma(\frac{x}{\epsilon})$. For all $\e>0,$ we have
\begin{equation}
  \e^{(1+\a)/2}\mathcal{D}\l(\Theta \mathcal{D}^*\varsigma^\e\r)=\e^{(1-\a)/2}\mathcal{V}^\e
\end{equation}
Thus, for any $u\in{H_\rho^{\a/2}}(D),$ we have
$$\l(\e^{(1-\a)/2}\mathcal{V}^\e u, v\r)=\e^{(1+\a)/2}\int_{D\times D}\Theta \mathcal{D}^*\varsigma
\mathcal{D}^*(uv)dxdz=\int_{D\times D}\Theta (\mathcal{D}_y^*\varsigma)^\e
\mathcal{D}^*(uv)dxdz$$
for all $v\in \mathcal{M}(D)$($\mathcal{M}(D)$ is the space of functions in $C^\infty(D)$ with compact supports). From the nonlocal Poincar\'{e} inequality\cite{Fel}, we have
$$\l|\l(\e^{(1-\a)/2}\mathcal{V}^\e u, v\r)\r|\leq C\l\|u\r\|_{H_\rho^{\a/2}(D)}\l\|v\r\|_{H_\rho^{\a/2}(D)},$$
where $c$ is a positive constant.
Thus, by the density of $\mathcal{M}(D)$ in $H_\rho^{\a/2}(D)$, the precedent inequality holds for all $v\in H_\rho^{\a/2}(D)$. Hence, inequality $(\ref{321})$ follows.
\end{remark}
Now, let us prove some uniform estimates in the following lemma. Before this lemma, we make an useful remark.
Let us put
  $$a^\e(u,v)=\int_{D}\int_{D}\Theta^\e\mathcal{D}^*u(x,z)\overline{\mathcal{D}^*v}(x,z)dxdz.$$
\begin{lemma}\label{l2}
  Let $u^\e$ be a solution of equation (\ref{1}) with initial value $h\in L^2(D).$ Suppose further that $$f, f'\in L^2\l(0,T;L^2(D)\r)$$and
  $$\l\|\e^{(-1-\a)/2}(\frac{\partial\mathcal{V}}{\partial\tau})^\e\r\|_{\mathcal{L}(H_\rho^{\a/2}(D), H_\rho^{-\a/2}(D))}\leq c_0,$$
  where $f'$ stands for $\frac{df}{dt}$, $c_0$ is a constant independent of $\epsilon.$ Then there exists a constant $c>0$ independent of $\epsilon$ such that the solution $u_\e$
  of equation (\ref{1}) verifies:
  $$\sup_\epsilon\mathbb{E}\sup_{0\leq t\leq T}\l\|u_\epsilon\r\|^2_{L^2(D)}+\sup_\epsilon\mathbb{E}\l\|u_\e\r\|^2_{L^2\l(0,T;H_\rho^{\a/2}(D)\r)}
  +\sup_\epsilon\mathbb{E}\l\|u'_\e\r\|^2_{L^2\l(0,T;H_\rho^{-\a/2}(D)\r)}\leq c. $$

\begin{proof}
  Applying It\^{o} formula for $u_\e(t),$ we have
  \begin{small}

  \begin{equation*}
  \begin{split}
  &\l\|u_\e(t)\r\|^2_{L^2}=\l\|u_\e(0)\r\|^2_{L^2}-\mbox{Re}\int_0^t2i\l(\mathcal{A}^\e u_\e(s)+\e^{(1-\a)/2}\mathcal{V^\e}u_\e(s),,u_\e(s)\r)ds\\
 &\quad-\mbox{Re}\int_0^t2i\l(g(u_\e(s)),u_\e(s)\r)dW_s-\mbox{Re}\int_0^t2i\l(f,u_\e(s)\r)ds+\int_0^t\l\|g(u_s)\r\|^2_{L^2}ds\\
 &=\l\|u_\e(0)\r\|^2_{L^2}+\mbox{Im}\int_0^t2\l(g(u_\e(s)),u_\e(s)\r)dW_s+\mbox{Im}\int_0^t2\l(f,u_\e(s)\r)ds+\int_0^t\l\|g(u_s)\r\|^2_{L^2}ds.
  \end{split}
  \end{equation*}
  \end{small}
  By Burkholder-Davis-Gundy's inequality, H\"{o}lder inequality and Young's inequality, it refers that
  \begin{equation*}
  \begin{split}
  \mathbb{E}&\sup_{0\leq t\leq T}2\mbox{Im}\int_0^t2\l(g(u_\e(s)),u_\e(s)\r)dW_s\\
  &=\mathbb{E}\sup_{0\leq t\leq T}2\mbox{Im}\int_0^t2\int_{D}g\l(u_\e(s)\r)\bar{u}_\e(s)dW_sdx\\
  &\leq c_1\mathbb{E}\l(\int_0^T\l\|\bar{u}_\e(s) g(u_\e(s))\r\|^2_{L^2}ds\r)^{\frac{1}{2}}\\
  &\leq c_1\mathbb{E}\l(\delta\sup_{0\leq t\leq T}\l\|u_\e(t)\r\|^2_{L^2}+\frac{1}{\delta}\int_0^T\l\|g(u_\e(s))\r\|^2_{L^2}ds\r)\\
  &\leq\frac{1}{3}\mathbb{E}\sup_{0\leq t\leq T}\l\|u_\e(t)\r\|^2_{L^2}+c_2\mathbb{E}\int_0^T\l\|u_\e(s)\r\|^2_{L^2}ds+c_2.
   \end{split}
  \end{equation*}
  Then, we obtain
 $$\frac{2}{3}\mathbb{E}\sup_{0\leq t\leq T}||u_\e(t)||^2_{L^2}\leq ||u_\e(0)||^2_{L^2}+c_3\int_0^T\sup_{0\leq r\leq s}||u_\e(r)||^2_{L^2}dr+c_4,$$
 which implies from Gronwall inequality that
 $$\mathbb{E}\sup_{0\leq t\leq T}||u_\e(t)||^2_{L^2}\leq c_5,$$
 where the positive constant $c_5$ is indepedent of $\e.$\\
 Moreover, we also have $$\mathbb{E}\sup_{0\leq t\leq T}||u_\e(t)||^4_{L^2}\leq c_5.$$
 Next, we denote $d\dot{W_t}=\frac{dW_t}{dt}$ and take the product in $L^2(D)$ of equation (\ref{1}) with $u'_\e$
 \begin{small}
 $$i\l\|u'_\e(t)\r\|^2_{L^2}=\l(\mathcal{A}^\e u_\e(t)+\e^{(1-\a)/2}\mathcal{V^\e}u_\e(t),,u'_\e(t)\r)+\l(g(u_\e(t)),u'_\e(t)\r)d\dot{W_t}+\l(f(t),u'_\e(t)\r).$$
 \end{small}
 By the preceding equality we have
 \begin{small}
 $$\mbox{Re}\l(\mathcal{A}^\e u_\e(t)+\e^{(1-\a)/2}\mathcal{V^\e}u_\e(t),u'_\e(t)\r)+\mbox{Re}\l(g(u_\e(t)),u'_\e(t)\r)d\dot{W_t}+\mbox{Re}\l(f(t),u'_\e(t)\r)=0.$$
 \end{small}
 Since the fact that
 $$\frac{d}{dt}a^\e\l(u_\e(t),u_\e(t)\r)=2\mbox{Re}\l(\mathcal{A}^\e u_\e(t),u'_\e(t)\r),$$
 \begin{small}
 $$\e^{(1-\a)/2}\frac{d}{dt}\l(\mathcal{V}^\e u_\e(t),u_\e(t)\r)=\e^{(-1-\a)/2}\l((\frac{\partial\mathcal{V}}{\partial{\tau}})^\e u_\e(t),u_\e(t)\r)
 +2{\e}^{(1-\a)/2}\mbox{Re}\l(\mathcal{V}^\e u_\e(t),u'_\e(t)\r).$$
 \end{small}
 We have
 \begin{equation*}
 \begin{split}
   &\frac{1}{2}\frac{d}{dt}a^\e(u_\e(t),u_\e(t))+\frac{1}{2}\e^{(1-\a)/2}\frac{d}{dt}(\mathcal{V}^\e u_\e(t),u_\e(t))-
   \e^{(-1-\a)/2}\l(\l(\frac{\partial\mathcal{V}}{\partial{\tau}}\r)^\e u_\e(t),u_\e(t)\r)\\
   &+\mbox{Re}\l(g(u_\e(t)),u'_\e(t)\r)d\dot{W_t}+\mbox{Re}\frac{d}{dt}(f(t),u_\e(t))
   -\mbox{Re}\l(f'(t),u_\e(t)\r)=0.
 \end{split}
 \end{equation*}
 An integration on $[0,t]$ of the equality above yields,
 \begin{small}
 \begin{equation*}
 \begin{split}
 &\frac{1}{2}a^\e(u_\e(t),u_\e(t))+\frac{1}{2}\e^{(1-\a)/2}\l(\mathcal{V}^\e u_\e(t),u_\e(t)\r)-\frac{1}{2}a^\e(h,h)
 -\frac{1}{2}\e^{(1-\a)/2}\l(\mathcal{V}^\e(0) h,h\r)\\
 &=\e^{(-1-\a)/2}\int_0^t\l(\l(\frac{\partial\mathcal{V}}{\partial{\tau}}\r)^\e u_\e(s),u_\e(s)\r)ds
 -\int_0^t\mbox{Re}\l(g(u_\e(s)),u'_\e(s)\r)dW_s\\
 &\quad-\mbox{Re}(f(t),u_\e(t))+\mbox{Re}(f(0),h)+\mbox{Re}\int_0^t\l(f'(s),u_\e(s)\r)ds.
\end{split}
 \end{equation*}
 \end{small}
 It follows that
  \begin{equation*}
 \begin{split}
  c_6||u_\e(t)||^2_{H_\rho^{\a/2}}&+2\int_0^t\mbox{Re}\l(g(u_\e(s)),u'_\e(s)\r)dW_s\\
  &\leq \beta||u_\e(t)||^2_{L^2}+c_7||h||^2_{H_\rho^{\a/2}}
  +\beta||h||^2_{L^2}\\&+c_0||u_\e(t)||^2_{L^2(0,T;L^2(D))}
  +2||f(t)||_{L^2}||u_\e(t)||_{L^2}\\
  &+2||f(0)||_{L^2}||h||_{L^2}+2||f'||_{L^2(0,T;L^2(D))}||u_\e(t)||_{L^2(0,T;L^2(D))}.
 \end{split}
 \end{equation*}
 We consider the expectation after integrating on $[0,T]$ the preceding inequality and using Burkholder-Davis-Gundy's inequality, we have
 $$\mathbb{E}||u_\e||^2_{L^2(0,T;H_\rho^{\a/2}(D))}\leq c_8,$$
  where the positive constant $c_8$ is indepedent of $\e.$
  By equatioin (\ref{1}), we have
  \begin{equation*}
 \begin{split}
  i\int_0^T(u'_\e(t),\bar{v}(t))dt&=\int_0^Ta^\e(u_\e(t),v(t))dt+\int_0^T\e^{(1-\a)/2}\l(\mathcal{V}^\e u_\e(t),v(t)\r)dt\\
  &\quad+\int_0^T(g(u_\e),v(t))dW_t+\int_0^T(f(t),v(t))dt
   \end{split}
 \end{equation*}
 for all $v\in L^2\l(0,T; H_\rho^{\a/2}(D)\r).$
 Hence, we have
 $$\mathbb{E}||u'_\e||^2_{L^2\l(0,T;H_\rho^{-\a/2}(D)\r)}\leq c_9.$$
 In summary, we deduce that
 \begin{small}
 $$\sup_\epsilon\mathbb{E}\sup_{0\leq t\leq T}||u_\epsilon||^2_{L^2(D)}+\sup_\epsilon\mathbb{E}||u_\e||^2_{L^2\l(0,T;H_\rho^{\a/2}(D)\r)}
  +\sup_\epsilon\mathbb{E}||u'_\e||^2_{L^2\l(0,T;H_\rho^{-\a/2}(D)\r)}\leq c. $$
  \end{small}
\end{proof}
\end{lemma}

\section{Effective Approximation and Effective System}
In this section, we will prove several convergence results. Then, we establish effective approximation and derive the effective system.
\subsection{Some convergence results}
Let us first introduce some functions spaces. Let $Q=D\times(0,T)$ with $T\in \mathbb{R}_+^*.$
We consider the space
$$\mathcal{Y}(0,T)=\l\{\upsilon\in L^2\l(\Omega\times (0,T);H^{\a/2}_\rho(D)\r): \upsilon'\in L^2\l(\Omega\times (0,T);H_\rho^{-\a/ 2}(D)\r)\r\}$$
provided with the norm
$$\|\upsilon\|^2_{\mathcal{Y}(0,T)}=\mathbb{E}\|\upsilon\|^2_{L^2\l(0,T;H^{\a/2}_\rho(D)\r)}+\mathbb{E}||\upsilon'||^2_{L^2\l(0,T ;H_\rho^{-\a/ 2}(D)\r)}$$
which makes it a Hilbert space.

\begin{definition}
  Let $E$ be a fundamental sequence. A sequence $(u_\e)_{\e\in E}\subset L^2(Q\times\Omega)$ is said to two-scale converge in $L^2(Q\times\Omega)$ to some $\tilde{u}\in L^2(Q\times\Omega; L^2_{per}(Y\times Z))$ if as
  $E\ni\e\rightarrow0,$
  \begin{small}
  $$\int_{Q\times\Omega}u_\e(x,t,\omega)\psi^\e(x,t,\omega)dxdtd{\mathbb{P}}\!\rightarrow\!\int_{Q\times\Omega\times Y\times Z}\tilde{u}(x,t,y,\tau,\omega)\psi(x,t,y,\tau,\omega))dxdtdyd\tau d{\mathbb{P}}$$
  \end{small}
  for all $\psi\in L^2(Q\times\Omega;\mathcal {C}_{per}(Y\times Z)),$ where $\psi^\e(x,t)=\psi(x,t,\frac{x}{\e},\frac{t}{\e},\omega).$
\end{definition}

\begin{lemma}\label{l3}
  Let $E$ be a fundamental sequence. $u^\e$ is a solution of equation (\ref{1}). Then, a subsequence $E'$ can be extracted from $E$ such that, as $E'\ni\e\rightarrow 0,$
  $$u_\e\rightarrow \tilde{u} \quad \text{in} \quad \mathcal{Y}(0,T)\text{-weakly}, $$
  \begin{center}
     $u_\e\rightarrow \tilde{u} \quad \text{in} \quad L^2(Q\times\Omega)\text{-two-scale}.$
  \end{center}
  Moreover, 
for a further subsequence $\e'\in E''$, we have
$$\mathcal{D}^*u_{\e'}\rightarrow D^*_x\tilde{u}+D^*_y u_1 \quad \text{in} \quad L^2(Q\times D\times\Omega) \text{-two-scale},$$
where $\tilde{u}\in \mathcal{Y}(0,T), u_1\in L^2(Q\times\Omega;L^2_{per}(Z;H_{\#}^{\a/2}(Y)).$
\begin{proof}
Let $$\psi_\e=\psi_0+\epsilon^{(1+\a)/2}\psi_1^\epsilon,\; i.e.,\;
\psi_\e(x,t,\omega)=\psi_0(x,t,\omega)+\epsilon^{(1+\a)/2}\psi_1(x,t,\frac{x}{\epsilon},\frac{t}{\epsilon},\omega),$$
 where $$\psi_0\in \mathcal{M}(Q)\otimes L^2(\Omega) \;and\; \psi_1\in \mathcal{M}(Q)\otimes[(\mathcal{C}_{per}(Y)/\mathbb{C})\otimes\mathcal{C}_{per}(Z)]\otimes L^2(\Omega).$$
We set $$\phi_0(x,z,t,\omega)=(\overline{D^*_x\psi_0})(x,z,t,\omega),$$
 $$\phi_1(x,\frac{x}{\epsilon'},\frac{z}{\epsilon'},t,\frac{t}{\epsilon'},\omega)=(\overline{D^*_y \psi_1})(x,\frac{x}{\epsilon'},\frac{z}{\epsilon'},t
 ,\frac{t}{\epsilon'},\omega),$$ $$\phi(x,z,t,y,\eta,\tau,\omega)=\phi_0(x,z,t,\omega)+\phi_1(x,y,t,\eta,\tau,\omega).$$
 For convenient, we omit variables $\omega$ and $t.$ The functions $\phi_0(x,z,t,\omega)$ and $\phi_1(x,\frac{x}{\epsilon'},\frac{z}{\epsilon'},t,\frac{t}{\epsilon'},\omega)$ are abbreviated as $\phi_0(x,z)$ and $\phi_1(x,\frac{x}{\epsilon'},\frac{z}{\epsilon'})$ respectively.
 Due to Lemma \ref{l2}, one has a subsequence $E,$ such that
   $$u_\e\rightarrow \tilde{u} \quad \text{in} \quad \mathcal{Y}(0,T)\text{-weakly}.$$
 Then for a further subsequence $E''\ni \e',$ we have
\begin{equation*}
\begin{cases}
u_{\e'} \quad \text{two-scale \;converges\; to} \;u\in L^2(Q\times\Omega\times Y\times Z),\\
\mathcal{D}^*u_{\e'}\text{\quad  two-scale\; converges\; to }\;U\in L^2(Q\times D\times\Omega\times Y\times Z\times N).
\end{cases}
 \end{equation*}
Hence, for any $\psi_{\epsilon'}$. one has
\begin{small}
\begin{equation}\label{e2}
\begin{split}
 &\int_0^T\int_{\Omega}a^{\e'}(u_{\epsilon'},\psi_{\epsilon'})dtd{\mathbb{P}}=\int_0^T\int_{\Omega}\int_{D}\int_{D}\Theta^{\e'}\mathcal{D}^*u_{\epsilon'}(x,z)
 \overline{\mathcal{D}^*\psi_{\epsilon'}}(x,z)dxdzdtd{\mathbb{P}}\\
 &\rightarrow \int_{Q\times D\times\Omega\times Y\times N \times Z}\Theta(y,\eta)U(x,z,t,y,\eta,\tau,\omega)\phi(x,z,t,y,\eta,\tau,\omega)dxdzdtdyd\eta d\tau d{\mathbb{P}}.
\end{split}
\end{equation}
\end{small}
By the definition of $\mathcal{D}^*$ and $\mathcal{D},$ it follows that
\begin{small}
\begin{equation*}
\begin{split}
&\int_0^T\int_{\Omega}a^{\epsilon'}(u_{\epsilon'},\psi_{\epsilon'})dtd{\mathbb{P}}\\
&=\int_{Q\times D}\int_{\Omega}\Theta^{\epsilon'}(x,z)(\mathcal{D}^*u_{\e'})(x,z)
\l[(\overline{D^*_x\psi_0})(x,z)+(\overline{D^*_y \psi_1})(x,\frac{x}{\epsilon'},\frac{z}{\epsilon'})\r]dxdzdtd{\mathbb{P}}+o(\epsilon')\\
&=\int_{Q\times D}\int_{\Omega}(\mathcal{D}^*u_{\e'})(x,z)\l[\phi_0(x,z)\Theta^{\epsilon'}(x,z)
+\phi_1((x,\frac{x}{\epsilon'},\frac{z}{\epsilon'}))\Theta^{\epsilon'}(x,z)\r]dxdzdtd{\mathbb{P}}+o(\epsilon')\\
&=\Lambda_1^{\epsilon'}+\Lambda_2^{\epsilon'}+o(\epsilon').
 \end{split}
 \end{equation*}
\end{small}
 For the first part of the right side, on one hand,
 \[
\begin{split}
 \Lambda_1^{\epsilon'}&=\int_{Q\times D}\int_{\Omega}(\mathcal{D}^*u_{\e'})(x,z)
  \phi_0(x,z)\Theta^{\epsilon'}(x,z)dxdzdtd{\mathbb{P}}\\
  &=\int_{Q\times D}\int_{\Omega}u_{\e'}(x)\l[\phi_0(x,z)\Theta^{\epsilon'}(x,z)+\phi_0(z,x)\Theta^{\epsilon'}(z,x)\r]\gamma(x,z)dzdxdtd{\mathbb{P}}.
\end{split}
 \]
 Let $\epsilon'$ goes to $0$, we have
 \begin{small}
 \[
\begin{split}
  \Lambda_1^{\epsilon'}&\rightarrow\int_{Q\times D\times\Omega\times Y\times N\times Z}u(x,y)\Theta(y,\eta)[\phi_0(x,z)+\phi_0(z,x)]\gamma(x,z)dzdxdtd\eta dyd\tau d{\mathbb{P}}\\
  &=\int_{Q\times\Omega\times Y\times N\times Z}u(x,y)\Theta(y,\eta)(D_x\phi_0)(x)dxdtdyd\eta d\tau d{\mathbb{P}}.
\end{split}
\]
\end{small}
On the other hand, from the fact that $\mathcal{D}^*u_{\e'}$  two-scale converges to $U\in L^2(Q\times D\times\Omega\times Y\times Z\times N)$, we have
\begin{equation*}
\begin{split}
    \Lambda_1^{\epsilon'}&=\int_{Q\times  D}\int_{\Omega}\Theta^{\epsilon'}(x,z)\mathcal{D}^*u_{\e'}(x,z)
(\overline{D^*_x\psi_0})(x,z)dxdzdtd{\mathbb{P}}\\
&\rightarrow \int_{Q\times D\times\Omega\times Y\times N \times Z}\Theta(y,\eta)U(x,z,t,y,\eta,\tau)\phi_0(x,z)dxdzdtdyd\eta d\tau d{\mathbb{P}}.
\end{split}
\end{equation*}
Then we have
\begin{small}
\begin{equation}\label{100}
  \int_{Q\times D\times\Omega\times Y\times N \times Z}\Theta(y,\eta)(U(x,z,t,y,\eta,\tau)-\mathcal{D}_x^*u(x,y))\phi_0(x,z)dxdzdtdyd\eta d\tau d{\mathbb{P}}=0.
\end{equation}
\end{small}
For the second part,
\begin{small}
\[
\begin{split}
\Lambda_2^{\epsilon'}&=\int_{Q\times\Omega}\int_{D}(\mathcal{D}^*u_{\e'})(x,z)\Theta^{\epsilon'}(x,z)
\phi_1(x,\frac{x}{{\epsilon'}},\frac{z}{{\epsilon'}})dxdzdtd{\mathbb{P}}\\
&=\int_{Q\times\Omega}u_{\e'}(x)\int_{D}\l[\Theta^{\epsilon'}(x,z)
\phi_1(x,\frac{x}{{\epsilon'}},\frac{z}{{\epsilon'}})+\Theta^{\epsilon'}(z,x)
\phi_1(z,\frac{z}{{\epsilon'}},\frac{x}{{\epsilon'}})\r]\gamma(x,z)dxdzdtd{\mathbb{P}}\\
&=\int_{Q\times\Omega}u_{\e'}(x)\int_{D}\l[\Theta^{\epsilon'}(x,z)
\phi_1(x,\frac{x}{{\epsilon'}},\frac{z}{{\epsilon'}})+\Theta^{\epsilon'}(z,x)
\phi_1(x,\frac{z}{{\epsilon'}},\frac{x}{{\epsilon'}})\r]\gamma(x,z)dxdzdtd{\mathbb{P}}\\
\quad &+\int_{Q\times\Omega}u_{\e'}(x)\int_{D}\l[\Theta^{\epsilon'}(z,x)
\phi_1(z,\frac{z}{{\epsilon'}},\frac{x}{{\epsilon'}})-\Theta^{\epsilon'}(z,x)
\phi_1(x,\frac{z}{{\epsilon'}},\frac{x}{{\epsilon'}})\r]\gamma(x,z)dxdzdtd{\mathbb{P}}\\
&=\Lambda_3^{\epsilon'}+\Lambda_4^{\epsilon'},
\end{split}
\]
\end{small}
where
\[
\begin{split}
\Lambda_3^{\epsilon'}
&={\epsilon'}^{(1-\a)/2}\int_{Q\times\Omega}u_{\e'}(x)[D_y(\Theta\phi_1)(x)]^{\e'} dxdtd{\mathbb{P}},
\end{split}
\]
and
\[
\begin{split}
\Lambda_4^{\epsilon'}&={\epsilon'}^{(1+\a)/2}\int_{Q\times\Omega}u_{\e'}(x)\int_{D}\Theta^{\epsilon'}(z,x)
[\psi_1(z,\frac{x}{{\epsilon'}})-\psi_1(x,\frac{x}{{\epsilon'}})\\
&\quad+\psi_1(x,\frac{z}{{\epsilon'}})
-\psi_1(z,\frac{z}{{\epsilon'}})]\gamma^2(x,z)dzdxdtd{\mathbb{P}}\\
&\rightarrow 0.
\end{split}
\]
We have
\begin{small}
\[
\begin{split}
\lim\limits_{\epsilon'\rightarrow 0}\int_0^T\int_{\Omega}a^{{\e'}}(u_{\e'},\psi_{\e'})dtd\mathbb{P}
&=\int_{Q\times\Omega\times Y\times N\times Z}u(x,y)\Theta(y,\eta)(D_x\phi_0)(x)dxdtdyd\eta d\tau d\mathbb{P}\\
&\quad+\lim\limits_{\epsilon'\rightarrow 0}\epsilon'^{(1-\a)/2}\int_{Q\times\Omega}u_{\e'}(x)[D_y(\Theta\phi_1)(x)]^{\e'} dxdtd\mathbb{P}.
\end{split}
\]
\end{small}
 By the two-scale convergence of $u_{\e'},$
 $$\int_{Q\times\Omega} \int_{Y \times Z}u(x,y)D_y(
 \Theta\phi_1)(x,t,y,\tau)dxdtdyd\tau d\mathbb{P}=0.$$
 This yields in paticular for any $\phi$
 $$\int_{Q\times\Omega\times Y\times N \times Z}(\mathcal{D}_y^*u)(x,y)(\Theta\phi_1)(x,t,y,\eta,\tau)dxdtdyd\eta d\tau d\mathbb{P}=0,$$
 hence,
 $$(\mathcal{D}_y^*u)(x,y)=0,$$
 which means that $u$ does not depend on $y.$ Then $u=\tilde{u}.$
 Now, we set $D_y(\Theta\phi_1)=0,$ we get
 \begin{small}
  \begin{equation*}
   \begin{split}
  \lim\limits_{\epsilon'\rightarrow 0}&\int_0^T\int_{\Omega}a^{{\e'}}(u_{\e'},\psi_{\e'})dtd\mathbb{P}
  =\int_{Q\times\Omega \times Y\times N\times Z}u(x,y)\Theta(y,\eta)(D_x\phi)(x)dxdtdyd\eta d\tau d\mathbb{P}\\
  &=\int_{Q\times D\times\Omega\times Y\times N \times Z}\Theta(y,\eta)D^*_x\tilde{u}(x,t,z)\phi(x,t,z,y,\eta,\tau)dzdxdtdy d\eta d\tau d\mathbb{P}.
  \end{split}
 \end{equation*}
 \end{small}
 We get that
 \begin{footnotesize}
 $$\int_{Q\times D\times\Omega \times Y\times N \times Z}\l(U(x,t,z,y,\tau,\eta)-D^*_x\tilde{u}(x,t,z)\r)\Theta(y,\eta)\phi(x,t,z,y,\eta,\tau)dzdxdtdy d\eta d\tau d\mathbb{P}=0.$$
 \end{footnotesize}
 From the formula (\ref{100}), we deduce that
  \begin{footnotesize}
  $$\int_{Q\times D\times\Omega\times Y\times N \times Z}(U(x,t,z,y,\tau,\eta)-D^*_x\tilde{u}(x,t,z))\Theta(y,\eta)\phi_1(x,t,y,\eta,\tau)dzdxdtdy d\eta d\tau d\mathbb{P}=0.$$
 \end{footnotesize}
 Since the fact that $D_y(\Theta\phi_1)=0,$ we can deduce that there exists a unique function $u_1\in L^2(Q\times\Omega;L_{per}^2(Z;H_{\#}^{\a/2}(Y)))$ such that
 $$U(x,t,z,y,\tau,\eta,\omega)-(D^*_x\tilde{u})(x,t,z,\omega)=(D^*_yu_1)(x,t,y,\tau,\eta,\omega).$$
 This ends the proof of Lemma \ref{l3}.
\end{proof}
\end{lemma}
\begin{remark}
   This setting for the two scale convergence method has a very unique feature in that, the limit of the sequence depends on additional variable which does not appear in the weak limit.
\end{remark}
\begin{remark}
  If the limit  in Lemma \ref{l3} can be shown to be unique then convergence of the whole sequence occurs.
\end{remark}

\begin{lemma}\label{l4}
  For all $\psi_0\in \mathcal{M}(Q)\times L^2(\Omega),$ for the subsequence $E''$ in Lemma \ref{l3}, we have
  \begin{footnotesize}
  \begin{align*}
  \int_{Q\times\Omega} \e^{(1-\a)/2}u_\e\mathcal{V}^\epsilon(x,t)\psi_0(x,t,\omega) dxdtd\mathbb{P}&\rightarrow2\int_{Q\times\Omega\times Y\times Z}u_1(x,t,y,\tau)\psi_0\mathcal{V}(y,\tau)dxdtdyd\tau d\mathbb{P}.\\
  \end{align*}
  \end{footnotesize}
  \begin{proof}
  Since the fact that the mean value of $\mathcal{V}$ for spatial variable is null, the equation
\begin{equation*}
 \mathcal{D}_y(\mathcal{D}_y^*\xi)=\mathcal{V}
\end{equation*}
admits a unique solution $\xi$ in $L^2_{per}(Z;H_{\#}^{\a/2}(Y))$. Let $\xi^\e(x,t)=\xi(\frac{x}{\epsilon},\frac{t}{\epsilon})$. For all $\e>0,$
   we can conclude
  \begin{align*}
    \int_{Q\times\Omega} \epsilon^{(1-\a)/2}u_\e\mathcal{V}^\epsilon(x,t)\psi_0 dxdtd\mathbb{P}=\e^{(1-\a)/2}\int_{Q\times\Omega}u_\e\psi_0 [\mathcal{D}_y(\mathcal{D}_y^*\xi)]^\epsilon dxdtd\mathbb{P}.
  \end{align*}
  Since the fact that, for every function $\Phi^\epsilon$, we have $\mathcal{D}\Phi^\epsilon=\epsilon^{\frac{1-\a}{2}}(D_y\Phi)^\e.$
  Let $\Phi=\mathcal{D}_y^*\xi$, we have
  \begin{align*}
    &\quad\int_{Q\times\Omega} \epsilon^{(1-\a)/2}u_\e\mathcal{V}^\epsilon(x,t)\psi_0 dxdtd\mathbb{P}
    =\int_0^T\int_{\Omega}\int_{\mathbb{R}}u_\e\psi_0 \mathcal{D}(\mathcal{D}_y^*\xi)^\epsilon dxdtd\mathbb{P}\\
    &=\int_0^T\int_{\Omega}\int_{\mathbb{R}}\int_{\mathbb{R}}\mathcal{D}^*(u_\e\psi_0)(\mathcal{D}_y^*\xi)^\epsilon dxdzdtd\mathbb{P}\\
    &=\int_0^T\int_{\Omega}\int_{\mathbb{R}}\int_{\mathbb{R}}[\mathcal{D}^*(u_\e)\psi_0(x)+\mathcal{D}^*(\psi_0)u_\epsilon(z)](\mathcal{D}_y^*\xi)^\epsilon dxdzdtd\mathbb{P}\\
    &\rightarrow\int_0^T\int_{\Omega}\int_{D}\int_{D}\int_{Y\times N \times Z}(\mathcal{D}^*\tilde{u}+\mathcal{D}_y^*u_1)\psi_0\mathcal{D}_y^*\xi dtdxdzd\tau dyd\eta d\mathbb{P}\\
    &+\int_0^T\int_{\Omega}\int_{D}\int_{D}\int_{Y\times N \times Z}\mathcal{D}^*\psi_0\tilde{u}(z)\mathcal{D}_y^*\xi dtdxdzd\tau dyd\eta d\mathbb{P}\\
    &+2\int_0^T\int_{\Omega}\int_{D}\int_{D^c}\int_{Y\times N \times Z}\tilde{u}(x)\psi_0(x)\gamma(x,z)\mathcal{D}_y^*\xi dtdxdzd\tau dyd\eta d\mathbb{P}\\
    &=|D|\int_0^T\int_{\Omega}\int_{D}\int_{Y \times Z}u_1\psi_0\mathcal{V}(y,\tau)dtdxd\tau dyd\mathbb{P}\\
    &+2\int_{Y\times N \times Z}\mathcal{D}_y^*\xi d\tau dyd\eta\int_0^T\int_{\Omega}\int_{\mathbb{R}}\int_{\mathbb{R}}\tilde{u}(x)\psi_0(x)\gamma(x,z)dxdzdtd\mathbb{P}\\
    &=2\int_0^T\int_{\Omega}\l(\int_{Y\times Z}u_1\mathcal{V}(y,\tau)dyd\tau,\psi_0(x)\r)dtd\mathbb{P}.
  \end{align*}
  Hence the conclusion in this lemma follows.
  \end{proof}
\end{lemma}

\subsection{Effective approximation theorem}
In this section, we will verify the main result that gives the effective approximation of equation (\ref{1}) and effective system.

Let us introduce some functions spaces. We consider the space
$$\mathbb{F}_0^1=\mathcal{Y}(0,T)\times L^2(Q\times\Omega;L^2_{per}(Z;H_{\#}^{\a/2}(Y)),$$
provided with the norm
$$||u||^2_{\mathbb{F}_0^1}=||\tilde{u}||^2_{\mathcal{Y}(0,T)}+||u_1||^2_{L^2(Q\times\Omega;L^2_{per}(Z;H_{\#}^{\a/2}(Y))},$$
which makes it Hilbert space. We consider also the space
$$\mathcal{F}_0^\infty=\mathcal{M}(Q)\otimes L^2(\Omega)\times[\mathcal{M}(Q)\otimes L^2(\Omega)\otimes (\mathcal{C}_{per}(Y)/\mathbb{C})\otimes\mathcal{C}_{per}(Z)],$$
which is a dense subspace of $\mathbb{F}_0^1.$ For $\bf{u}$$=(\tilde{u},u_1)$ and $\bf{v}$ $=(v_0,v_1)\in L^2(\Omega;H_\rho^{\a/2}(D))\times L^2(D\times\Omega;L^2_{per}(Z;H_{\#}^{\a/2}(Y)),$ we set
\begin{center}
  $a(\bf{u},v)$$=\int_{D\times D\times Y\times N \times Z}\Theta(y,\eta)(D^*_x\tilde{u}+D^*_yu_1)(\overline{D^*_xv_0}+\overline{D^*_yv_1})dxdzdyd\eta d\tau.$
\end{center}

We give the uniqueness by the following lemma \cite{Gaw}.
\begin{lemma}\label{l5}
  Suppose $f\in L^2(0,T;L^2(D))$, the variational problem
  \begin{equation}\label{e5}
   \begin{cases}
    $$\textbf{u}=(\tilde{u},u_1)\in\mathbb{F}_0^1\; with\; \tilde{u}(0)=h$$\\
    \begin{split}
    i\int_0^T\int_{\Omega}\l<\tilde{u}'(t),\overline{v_0}(t)\r>dtd\mathbb{P}&=\frac{1}{2}\int_0^T\int_{\Omega}a(\textbf{u}(t),\textbf{v}(t))dtd\mathbb{P}\\
    &+\int_0^T\int_{\Omega}(g(\tilde{u}),v_0)dW_td\mathbb{P}+\int_0^T(f(t),v_0(t))dtd\mathbb{P}\\
    &+2\int_0^T\int_{\Omega}(\int_{Y\times Z}u_1\mathcal{V}(y,\tau)dyd\tau,v_0(x))_{L^2(D)}dtd\mathbb{P}\\
    &+\int_{Q}\int_{\Omega}\int_{Y\times N}\rho(x)\Theta(y,\eta)\tilde{u}(x)\overline{v_0}(x)dxdtdyd\eta d\mathbb{P}\\
    \text{for all} \;\textbf{v}=(v_0,v_1)\in\mathbb{F}_0^1\\
    \end{split}
   \end{cases}
  \end{equation}
  admits at most one solution($\l<,\r>$ is the dual pairing between $H_\rho^{\a/2}$ and $H_\rho^{-\a/2}$).
\end{lemma}

Next, we will show $\textbf{u}=(\tilde{u},u_1),$ where $\tilde{u},u_1$ is defined in Lemma \ref{l3}.
\begin{theorem}(Effective approximation)
Suppose the hypotheses of Lemma \ref{l1} and Lemma \ref{l3} are satisfied. For $\e>0,$ let $u_\e$ be the solution of equation (\ref{1}). Then, as $\e\rightarrow0,$ we have
\begin{equation}\label{110}
  u_\e\rightarrow \tilde{u} \quad \text{in} \quad \mathcal{Y}(0,T)\text{-weakly},
\end{equation}
  \begin{equation}\label{111}
  u_\e\rightarrow \tilde{u} \quad \text{in} \quad L^2(\l(0,T\r)\times D\times\Omega)\text{-strongly},
    \end{equation}
  where, $(\tilde{u},u_1)=\textbf{u} \in \mathbb{F}_0^1$ is the unique solution of equation (\ref{e5}).
\begin{proof}

Thanks to the Lemma \ref{l3}, there are some subsequence $E'$ extracted from $E$ and some vector function $\textbf{u}=(\tilde{u},u_1)\in \mathbb{F}_0^1$ such that the convergence is satisfied when $E'\ni \e\rightarrow0.$

Thus, according to Lemma \ref{l5}, the theorem is certainly proved if we can show that $\textbf{u}$ vertifies equation (\ref{e5}).

Indeed, we begin by vertifying that $\tilde{u}(0)=h.$
Let $\upsilon\in H_\rho^{\a/2}$ and  $\varphi\in \mathcal{C}^1([0,T])$ with $\varphi(T)=0.$ By integration by parts, we have
$$\int_0^T\l<u'_\e(t),\upsilon\r>\varphi(t)dt+\int_0^T\l<u_\e(t),\upsilon\r>\varphi'(t)dt=-\l<h,\upsilon\r>\varphi(0),$$
we pass to the limit in the preceding equality as $\e\rightarrow0.$ we obtain
$$\int_0^T\l<\tilde{u}'(t),\upsilon\r>\varphi(t)dt+\int_0^T\l<\tilde{u}(t),\upsilon\r>\varphi'(t)dt=-\l<h,\upsilon\r>\varphi(0).$$
Since $\varphi$ and $\upsilon$ are arbitary, we see that $\tilde{u}(0)=h.$
Finally, let us prove the variational equality of (\ref{e5}).
We let $\psi^{\e}\in L^2(Q\times\Omega;\mathcal {C}_{per}(Y\times Z)),$ then there are two functions
$$\psi_0\in \mathcal{M}(Q)\otimes L^2(\Omega) \;\text{and}\; \psi_1\in \mathcal{M}(Q)\otimes L^2(\Omega)\otimes[(\mathcal{C}_{per}(Y)/\mathbb{C})\otimes\mathcal{C}_{per}(Z)],$$
such that
$$\psi_\e=\psi_0+\epsilon^{(1+\a)/2}\psi_1^\epsilon,\; i.e.,\;
\psi_\e(x,t,\omega)=\psi_0(x,t,\omega)+\epsilon^{(1+\a)/2}\psi_1(x,t,\frac{x}{\epsilon},\frac{t}{\epsilon},\omega).$$
By equation (\ref{1}), one as
\begin{small}
 \begin{equation}\label{e8}
 \begin{split}
  i\int_0^T\int_{\Omega}&\l<u'_\e(t),\bar{\psi}_\e(t)\r>dtd\mathbb{P}=\frac{1}{2}\int_0^T\int_{\Omega}a^\e(u_\e(t),\psi_\e(t))dtd\mathbb{P}+\int_0^T
  \int_{\Omega}(f(t),\psi_\e(t))dtd\mathbb{P}\\
  &+\int_0^T\int_{\Omega}(g(u_\e),\psi_\e(t))dW_td\mathbb{P}
  +\int_0^T\int_{\Omega}(\e^{(1-\a)/2}\mathcal{V}^\e u_\e(t),\psi_\e(t))dtd\mathbb{P}\\
  &+\int_Q\int_{\Omega}\Theta^\e(x,z)\rho(x)u_\e(t)\psi_\e(t)dxdzdtd\mathbb{P}.
   \end{split}
 \end{equation}
 \end{small}
The aim is to pass to the limit in the above equation as $\e$ goes to $0.$ First, we have
$$\int_0^T\int_{\Omega}\l<u'_\e(t),\bar{\psi}_\e(t)\r>dtd\mathbb{P}=-\int_Q\int_{\Omega}u_\e\frac{\partial \bar{\psi}_\e}{\partial t}dxdtd\mathbb{P}.$$
Thus, we have
$$\int_0^T\int_{\Omega}\l<u'_\e(t),\bar{\psi}_\e(t)\r>dtd\mathbb{P}\rightarrow -\int_Q\int_{\Omega}\tilde{u}\frac{\partial \bar{\psi}_0}{\partial t}dxdtd\mathbb{P}=\int_0^T\int_{\Omega}\l<\tilde{u}'(t),\bar{\psi}_0(t)\r>dtd\mathbb{P},$$
as $\e\rightarrow0.$
Then, from Lemma \ref{l3}, as $\e$ goes to $0,$ we have
$$\int_0^T\int_{\Omega}a^\e(u_\e(t),\psi_\e(t))dtd\mathbb{P}\rightarrow\int_0^T\int_{\Omega} a(\textbf{u}(t),\boldsymbol{\phi}(t))dtd\mathbb{P},$$
where $\boldsymbol{\phi}=(\psi_0,\psi_1).$
Finally
\begin{footnotesize}
 \begin{equation}\label{e6}
  \int_0^T\int_{\Omega}\e^{(1-\a)/2}(\mathcal{V}^\e u_\e(t),\psi_\e(t))dtd\mathbb{P}=\e^{(1-\a)/2}\int_Q\int_{\Omega}\mathcal{V}^\e u_\e\overline{\psi}_0dxdtd\mathbb{P}+\epsilon\int_Q\int_{\Omega}
 \mathcal{V}^\e u_\e\overline{\psi}^\e_1dxdtd\mathbb{P}
 \end{equation}
 \end{footnotesize}

 In view of Lemma \ref{l4}, we pass to the limit in (\ref{e6}). This yields,

 \begin{equation*}
 \begin{split}
 \int_0^T\int_{\Omega}\e^{(1-\a)/2}(\mathcal{V}^\e u_\e(t),\psi_\e(t))dtd\mathbb{P}&\rightarrow\int_{\Omega\times Q\times Y \times N}u_1\overline{\psi}_0\mathcal{V}dxdtdyd\tau d\mathbb{P},\\
\end{split}
\end{equation*}

 as $\e\rightarrow0.$

 Hence, passing to the limit in (\ref{e8}) leads to
\begin{equation}\label{e9}
 \begin{split}
  i\int_0^T\int_{\Omega}&\l<u'_0(t),\bar{\psi}_0(t)\r>dtd\mathbb{P}=\frac{1}{2}\int_0^T\int_{\Omega}a(\textbf{u}(t),\boldsymbol{\phi}(t))dtd\mathbb{P}
  +\int_0^T\int_{\Omega}(f(t),\psi_0(t))dtd\mathbb{P}\\
  &\quad+\int_0^T\int_{\Omega}(g(\tilde{u}),\psi_0(t))dW_td\mathbb{P}  +\int_{\Omega\times Q\times Y \times N}u_1\overline{\psi}_0\mathcal{V}dxdtdyd\tau d\mathbb{P}\\
  &\quad+\int_Q\int_{\Omega}\int_{Y\times N}\rho(x)\Theta(y,\eta)\tilde{u}(x)\overline{\psi_0}(x)dxdtdyd\eta d\mathbb{P}\\
   \end{split}
 \end{equation}
for all $\boldsymbol{\phi}=(\psi_0,\psi_1)\in \mathcal{F}_0^\infty.$ Moreover, since $\mathcal{F}_0^\infty$ is a dense subspace of $\mathbb{F}_0^1,$
by (\ref{e9}) we see that $\textbf{u}=(\tilde{u},u_1)$ verifies (\ref{e5}). Thanks to the uniqueness of the solution for (\ref{e5}) and let the fact that the sequence $E$ is arbitrary, the theorem is proved.
\end{proof}
\end{theorem}

For further needs, we wish to give a simple representation of the function $u_1.$
For this purpose, Let us introduce the bilinear form $${\hat{a}}(w,v)=\int_{Y\times N \times Z}\Theta(y,\eta)(D^*_yw\cdot\overline{D^*_yv})dy d\eta d\tau $$
for all $w,v\in L^2_{per}(Z;H_{\#}^{\a/2}(Y)).$

Due to the nonlocal Poincar\'{e} inequality, we obtain ${\hat{a}}(w,v)$ is coercive on the space $L^2_{per}(Z;H_{\#}^{\a/2}(Y))$. Moreover,
${\hat{a}}(w,v)\leq C ||w||_{H_{\#}^{\a/2}(Y)}||v||_{H_{\#}^{\a/2}(Y)}, $ where $C$ is a positive constant.


\begin{remark}\label{r5}(Representation of $u_1$)
We consider the variational problem
\begin{equation}\label{654}
\begin{cases}
 $${\hat{a}}(\chi,v)=\int_{Y\times N \times Z}\Theta(y,\eta)\overline{D^*_yv}dyd\tau d\eta,\\
 \chi\in L^2_{per}(Z;H_{\#}^{\a/2}(Y)),\\
\end{cases}
\end{equation}
for all $v\in L^2_{per}(Z;H_{\#}^{\a/2}(Y)).$ It determines $\chi$ in a unique manner.

  Under the assumption of Lemma \ref{l3}, we have
  \begin{equation}\label{eyong17}
  u_1(x,t,y,\tau,\omega)=-\frac{1}{|D|}\int_{D}(D^*_x\tilde{u})(x,t,z,\omega)dz\cdot\chi(y,\tau)
  \end{equation}
   for almost all $(x,t,y,\tau,\omega)\in Q\times Y \times Z\times\Omega.$ \newline
In fact, in view of (\ref{e5}), we choose the particular test function $\textbf{v}=(v_0,v_1)\in \mathbb{F}_0^1$ with $v_0=0$ and $v_1=\varphi\times v,$ where
$\varphi\in \mathcal{M}(Q)$ and $v\in L^2_{per}(Z;H_{\#}^{\a/2}(Y)).$ This yields
\begin{equation}\label{103}
\begin{split}
  0&=|D|\int_{\Omega}{\hat{a}}(u_1,v)d\mathbb{P}
  +\int_{\Omega}\int_{D}(\mathcal{D}_x^*\tilde{u})(x,t,z)dzd\mathbb{P}\\
  &\quad\times\int_{Y\times N \times Z}(\mathcal{D}_y^*v)(x,t,y,\tau,\eta)\Theta(y,\eta)dyd\tau d\eta dx,
  \end{split}
\end{equation}
almost everywhere in $(x,t)\in Q$ and for all $v\in L^2_{per}(Z;H_{\#}^{\a/2}(Y)).$ By the fact that $u_1(x,t,\omega)($for fixed $(x,t,\omega)\in (Q\times\Omega)$) is the sole function in  $L^2_{per}(Z;H_{\#}^{\a/2}(Y))$ solving equation (\ref{103}). At the same time, it is an easy matter to check that the right hand side
of (\ref{eyong17}) solves the same equation (\ref{103}).
  \end{remark}


  \subsection{Effective system}\label{sub3.3}
  In this section, we will show  that the limit process $\tilde{u}$ satisfies the following nonlocal stochastic  Schr\"{o}dinger equation (\textbf{effective system}):
 \begin{align}\label{e7}
  \begin{cases}
    id\tilde{u}=\l(-\Xi_1(-\Delta)^{\a/2}\tilde{u}-\frac{\Xi_2}{2}(\mathcal{D}\zeta)(x)-\Xi_3\zeta(x)\r)dt+g(\tilde{u})dW_t+fdt,\\
    \tilde{u}(x,t)=0,\quad\quad (x,t) \in D^c\times(0,T),\\
\tilde{u}(0)=h(x),  \quad\quad x \in D,$$
  \end{cases}
\end{align}
where
\begin{align*}
  &\Xi_1=\int_{Y\times N}\Theta(y,\eta)dydn,\\
  &\Xi_2=\int_{Y\times N \times Z}\Theta(y,\eta)\mathcal{D}_y^*\chi dydnd\tau,\\
  &\Xi_3=2\int_{Y\times Z}\mathcal{V}(y,\tau)\chi(y,\tau)dyd\tau,\\
  &\zeta(x)=\frac{1}{|D|}\int_{D}(\mathcal{D}^*\tilde{u})(x,z)dz,\\
    &\mathcal{D}|_{D}(\zeta)(x)=\int_{D}(\zeta(x)+\zeta(z))\gamma(x,z)dz.\\
\end{align*}
We can see that if $\tilde{u}$ verifies equation (\ref{e7}) then $\textbf{u}=(\tilde{u},u_1)$ satisfies equation (\ref{e5}). From Lemma \ref{l5}, we obtain equation (\ref{e7}) has at most one weak solution $\tilde{u}.$

\begin{theorem}(Effective equation)
  Suppose the hypotheses of Lemma \ref{l1} and \ref{l2} are satisfied. Let $u_\e$ be the solution of heterogeneous equation (\ref{1}). Then, as scale parameter epsilon goes to $0$, we have $u_\e\rightarrow \tilde{u}$ in $\mathcal{Y}(0,T)$-weakly, where $\tilde{u}$ is the unique weak solution of effective equation (\ref{e7}) in $\mathcal{Y}(0,T)$.
\end{theorem}
\begin{proof}
Since the fact that, from any fundamental sequence $\e'\in E$ one can extract a subsequence $E'$ such that as $\e$ goes to $0$, we have (\ref{110})-(\ref{111}), and  (\ref{e9}) holds for all $\boldsymbol{\phi}=(\psi_0,\psi_1)\in \mathcal{F}_0^\infty$, where $\textbf{u}=(\tilde{u},u_1)\in \mathbb{F}_0^1.$ Now, substituting $u_1$ in Remark \ref{r5} to (\ref{e9}), a simple computation yields equation (\ref{e7}).
\end{proof}

\subsection{An example}
In this subsection, we set $\Theta^{\e}(x,y)=1.$ We can simplify the heterogeneous system (\ref{1}):
\begin{equation}
\begin{cases}
$$idu_\e=\l(-(-\Delta)^{\a/2} u_\e+\epsilon^{(1-\a)/2}\mathcal{V}^\epsilon u_\e\r)dt+g(u_\e)dW_t+fdt, \; t>0, \; x \in D=(-1,1),  \\
u_\epsilon(0,x)=h(x),\quad\quad x \in D,\\
u_\e(t,x)=0,  \quad\quad x \in D^c=\mathbb{R}\backslash D,$$
\end{cases}
\end{equation}

From subsection \ref{sub3.3}, equation(\ref{654}) and the fact that the function $\chi, \xi\in L^2_{per}(Z;H_{\#}^{\a/2}(Y))$, we have
\begin{align*}
  \Xi_1&=\int_{Y\times N}\Theta(y,\eta)dydn=1,\\
  \Xi_2&=\int_{Y\times N \times Z}\mathcal{D}_y^*\chi dydnd\tau=\int_Z\int_{Y\times N}\l(\chi(y,\tau)-\chi(\eta,\tau)\r)(y-\eta)\frac{1}{|y-\eta|^{\frac{3+\a}{2}}}dyd\eta d\tau\\
  &=2\int_Z\int_{Y\times N}\chi(y,\tau)(y-\eta)\frac{1}{|y-\eta|^{\frac{3+\a}{2}}}dyd\eta d\tau=0,\\
  \Xi_3&=\int_{Y\times Z}\mathcal{V}(y,\tau)\chi(y,\tau)dyd\tau=0.\\
  \end{align*}
Then, we can obtain the effective system
\begin{align}
  \begin{cases}
    id\tilde{u}=\l(-(-\Delta)^{\a/2}\tilde{u}\r)dt+g(\tilde{u})dW_t+fdt,\\
    \tilde{u}(x,t)=0,\quad\quad (x,t) \in D^c\times(0,T),\\
\tilde{u}(0)=h(x),  \quad\quad x \in D.$$
  \end{cases}
\end{align}
We can see that, in this situation, the term of oscillating potential  has no influence on the effective system.

\end{document}